\documentclass[12pt,oneside]{amsart}
\usepackage[margin=1in]{geometry}
\usepackage{amssymb}
\usepackage{amsmath}
\usepackage{amsthm}
\usepackage{amsfonts}
\usepackage[hidelinks]{hyperref}
\usepackage{bm}
\usepackage{tikz}
\usetikzlibrary{calc,decorations.markings}

\usepackage{mathrsfs}

\newenvironment{nouppercase}{%
  \renewcommand{\uppercasenonmath}[1]{}}{}

\numberwithin{equation}{section}
\numberwithin{figure}{section}

\theoremstyle{plain}
\newtheorem{thm}[equation]{Theorem}
\newtheorem{lemma}[equation]{Lemma}

\newtheorem{prop}[equation]{Proposition}

\theoremstyle{definition}

\newtheorem{remark}[equation]{Remark}

\begin{document}

\title[Trigonometric inequalities and the Riemann zeta-function]{Trigonometric inequalities and the Riemann zeta-function}

\author{Pace P.\ Nielsen}
\address{Department of Mathematics, Brigham Young University, Provo, UT 84602, USA}
\email{pace@math.byu.edu}

\begin{abstract}
We very slightly improve the leading constant of the (currently best) proven asymptotic zero-free region of the Riemann zeta-function, by using an easy improvement to a trigonometric polynomial.
\end{abstract}

\begin{nouppercase}
\maketitle
\end{nouppercase}

\section{Introduction}

In the process of working through the derivation of the (currently best) proven asymptotic zero-free region of the Riemann zeta-function, as obtained by Kevin Ford in \cite{Ford}, a very slight improvement in the leading constant, from $0.5507$ to $0.55127\ldots$, was obtained.  The improvement comes from considering a slightly larger class of nonnegative trigonometric polynomials; namely, allowing those with an extra factor of $1+\cos(\vartheta)$.

This paper was originally meant to serve only as a set of personal notes on optimizing the zero-free region of the Riemann zeta-function.  However, it is being posted to the arXiv in the hope that it may be of use to those who need even a slight improvement in the asymptotic range of the zero-free region.  Of course, the slightly more general trigonometric polynomials that we use here may also find use in future papers on zero-free regions.  It is hoped that non-experts will find the colloquial nature of this paper helpful, but experts may safely skip many of the details.  In fact, those familiar with the notation and terminology of \cite{Ford}, and who are only interested in the improved constant, may safely jump to the \emph{last paragraph} of the paper for those details, keeping \eqref{Eq:TrigInequalityBs} in mind.

The only other new pieces of information in this paper are as follows: (1) a reinterpretation of \cite[Lemma 5.1]{Ford}, using the midpoint approximation of an integral, that arises from our Lemma \ref{Lemma:NewkSums} and (2) correcting a few minor typos that appear in the literature.  Everything else is a rehashing of the computations and ideas already found in \cite{Ford}.

\section{Reinterpreting zeros}

The starting point of many computations of zero-free regions is Cauchy's residue formula.  Let $S,T>0$ be two large, real parameters.  Let $C_{S,T}$ be the counterclockwise rectangular contour whose corners are $\frac{1}{2}+\frac{1}{S} -iT$, $1+S-iT$, $1+S+iT$, and $\frac{1}{2}+\frac{1}{S}+iT$.  The Riemann hypothesis is equivalent to the claim that as $S,T\to \infty$ we obtain
\[
\frac{1}{2\pi i}\oint_{C_{S,T}} \frac{\zeta'(s)}{\zeta(s)}\, ds=-1.
\]
Only the pole at $s=1$ should contribute to the integral, as pictured below:
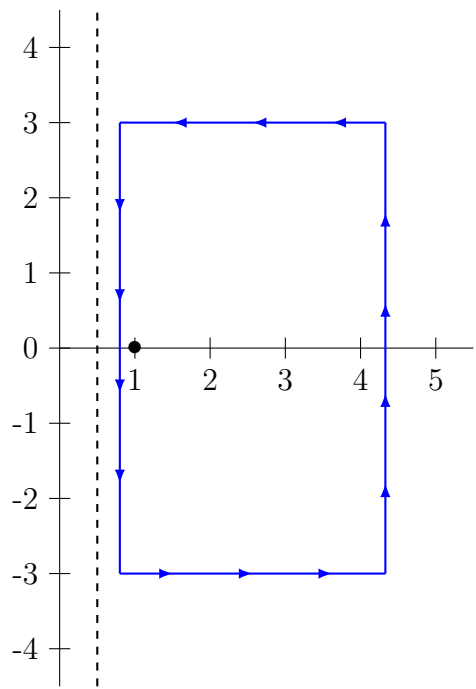
\begin{center}
\begin{tikzpicture}
\draw (0,-4.5) -- (0,4.5);  
\draw (0,0) -- (5.5,0);
\foreach \y in {-4,...,4} {
  \draw (-4pt,\y) -- (4pt,\y) node[pos=0,left] {\y};
}
\foreach \y in {1,...,5} {
  \draw (\y,-4pt) -- (\y,4pt) node[pos=0,below] {\y};
}

\node at (1,0) {$\bullet$}; 

\draw[dashed,thick] (0.5,-4.5) -- (0.5,4.5);
\draw[thick,blue,
decoration={ markings,
      mark=at position 0.2 with {\arrow{latex}},
      mark=at position 0.4 with {\arrow{latex}},
      mark=at position 0.6 with {\arrow{latex}},
      mark=at position 0.8 with {\arrow{latex}}},
      postaction={decorate}
]
  (0.8,3) -- (0.8,-3);
\draw[thick,blue,
decoration={ markings,
      mark=at position 0.2 with {\arrow{latex}},
      mark=at position 0.5 with {\arrow{latex}},
      mark=at position 0.8 with {\arrow{latex}}},
      postaction={decorate}
]
 (0.8,-3)--(4.33,-3);
\draw[thick,blue,
decoration={ markings,
      mark=at position 0.2 with {\arrow{latex}},
      mark=at position 0.4 with {\arrow{latex}},
      mark=at position 0.6 with {\arrow{latex}},
      mark=at position 0.8 with {\arrow{latex}}},
      postaction={decorate}
]
 (4.33,-3)--(4.33,3);
\draw[thick,blue,
decoration={ markings,
      mark=at position 0.2 with {\arrow{latex}},
      mark=at position 0.5 with {\arrow{latex}},
      mark=at position 0.8 with {\arrow{latex}}},
      postaction={decorate}
]
 (4.33,3)--(0.8,3);
\end{tikzpicture}
\end{center}

Using current knowledge of $\zeta(s)$, the right side of the rectangle is easy to manage, but the other sides are not.  To overcome this defect, instead of integrating over a very large rectangle, it has been standard procedure to instead integrate over a very small circle containing a purported zero that occurs to the right of the line $\Re(s)=1/2$.  The circle is chosen to barely penetrate the critical strip.  Writing that purported zero in the form $\beta+it$, for some real number $1/2 <\beta<1$, and some (moderately large) real number $t>0$, the contour integration looks like
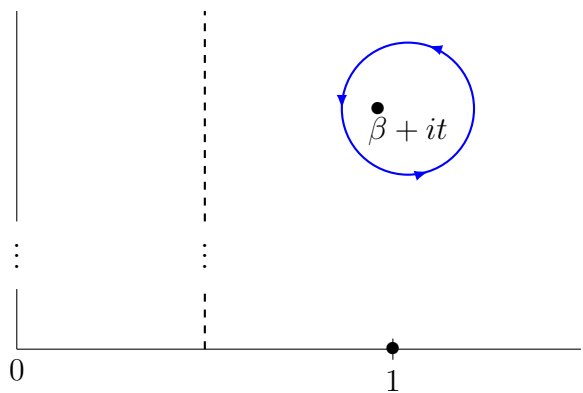
\begin{center}
\begin{tikzpicture}
\draw (0,0) -- (0,0.8);
\node at (0,1.35) {$\vdots$};
\draw (0,1.7) -- (0,4.5);
\draw (0,0) -- (7.5,0);

\node at (5,0) {$\bullet$}; 

\draw (0,0) node[below] {0};
\draw (5,4pt)--(5,-4pt) node[below] {1};
\draw[dashed,thick] (2.5,0)--(2.5,0.8);
\node at (2.5,1.35) {$\vdots$};
\draw[dashed,thick] (2.5,1.7)--(2.5,4.5);

\node at (4.8,3.2) {$\bullet$};
\node at (5.2,2.9) {$\beta+it$};

\draw[thick,blue,decoration={ markings,
      mark=at position 0.2 with {\arrow{latex}},
      mark=at position 0.5 with {\arrow{latex}},
      mark=at position 0.8 with {\arrow{latex}}},
      postaction={decorate}] (5.2,3.2) circle (25pt);
\end{tikzpicture}
\end{center}

In order to make the integral more tractable, the integral is not applied to $\zeta'(s)/\zeta(s)$ directly, but rather one integrates against a well-chosen function.  For instance, the chosen function of \cite[Lemma 3.2]{HB2} is $s^{-1}-sR^{-2}$, where $R$ is the radius of the circle.

A key insight of the paper \cite{Ford} is that by integrating against a properly chosen function, we may again use rectangular contours.  But what function should we use?  Four ideas help us choose wisely.  First, by choosing a function that quickly decreases to zero as the imaginary part increases, as well as deciding to keep the width of the rectangle fixed, then the contributions from the top and bottom sides of the rectangle become negligible as $T\to \infty$.  Second, think of the integral of $\zeta'(s)/\zeta(s)$ as equal to $\log\zeta(s)$; this introduces a branch cut along the contour.  This issue can be smoothed over by choosing a function with a zero along the contour, at an appropriately chosen point to simplify the computations.  Third, while we don't expect the left and right sides of the rectangle to behave exactly the same, we should add as much symmetry as possible to the picture.  Fourth, it is convenient to focus some attention to the center of the rectangle.  This is done by having a pole of the function at the center of the rectangle.

Putting all of these considerations together, then if our rectangle is centered at a point $z$, and the left and right sides are distance $\eta>0$ from $z$, then a function satisfying all of the needed conditions is naturally $h(s-z)$ where
\[
h(s)=\frac{\pi}{2\eta}\cot\left(\frac{\pi s}{2\eta}\right).
\]
This leads to the following general lemma.

\begin{lemma}[see {\cite[Lemma 2.2]{Ford}}]\label{Lemma:FordMainLemma}
Suppose $f$ is the quotient of two entire functions of finite order, and does not have a zero or a pole at $s=z$ nor at $s=0$.  Then, for all $\eta>0$ except for a set of Lebesgue measure $0$ \textup{(}the exception set may depend on $f$ and $z$\textup{)}, we have
\begin{equation}\label{Eq:FirstLemma}
\begin{array}{rcl}
\displaystyle-\Re\frac{f'(z)}{f(z)} &= &\displaystyle \frac{\pi}{2\eta}\sum_{|\Re(z-\rho)|\leq \eta}m_{\rho}\Re\cot\left(\frac{\pi(\rho-z)}{2\eta}\right)\\[20pt]
& &\displaystyle \quad + \frac{1}{4\eta}\int_{-\infty}^{\infty} \frac{\log|f(z-\eta+\frac{2\eta i u}{\pi})|-\log|f(z+\eta+\frac{2\eta i u}{\pi})|}{\cosh^2(u)}\, du,
\end{array}
\end{equation}
where $\rho$ runs over the zeros and poles of $f$ \textup{(}with multiplicity\textup{)}, and $m_{\rho}$ is either $1$ or $-1$ according to whether $\rho$ is a zero or a pole respectively.\hfill\qed
\end{lemma}

\begin{remark}
The two log terms in the numerator arise when evaluating the path integral along the left and right sides of the rectangle, respectively.  The other terms occur as the residues, coming from the poles of the integrand $\frac{f'(s)}{f(s)}h(s-z)$.  Exceptional values of $\eta$ occur when the integrals diverge, which might happen if too many zeros or poles are close to the two lines $\Re s=z\pm \eta$.

There is a small typo in the proof of this lemma.  In the second paragraph, when the integrals $I_j$ are being defined, a factor $\frac{1}{2\pi i}$ seems to be missing, but it is reintroduced later.
\end{remark}

\section{Simplifying considerations}

Lemma \ref{Lemma:FordMainLemma} applies when $f=\zeta$, since $\zeta$ has order $1$.  Due to well-known growth conditions on $\zeta$ in the vertical direction, there are no exceptional values of $\eta$; this is pointed out in the first two sentences of the proof of \cite[Lemma 3.4]{Ford}.  We can additionally simplify some of the quantities that appear in \eqref{Eq:FirstLemma} when $f=\zeta$, which will be the topic of this section.

The left side of \eqref{Eq:FirstLemma} has a nice representation as a Dirichlet series, but only when $\Re z>1$.  Moreover, we have very little knowledge of $-\zeta'/\zeta$ in the critical strip.  So, we will hereafter assume $\Re z>1$.

To take the best advantage of our purported zero $\rho_0=\beta+it$, we would like to take $z$ as close as possible to $\rho_0$.  Thus, the most important value of $z$ is $1+\epsilon+it$, where $\epsilon>0$ is a real parameter that we would like to send to zero.  However, we will see later that there are other values of $z$ that are also important.

Under the assumption that $\Re(z)>1$, it is well-known that every zero of $\zeta$ occurs strictly to the left of $z$. Given a zero, $\rho$, with $|\Re(z-\rho)|\leq \eta$, then we have $0< \Re(z-\rho)\leq \eta$.  This implies that
\[
\Re(h(\rho-z))=\frac{\pi}{2\eta}\Re\cot\left(\frac{\pi(\rho-z)}{2\eta}\right)\leq  0.
\]
Suppose we were to completely ignore the contributions of the zeros $\rho\neq \beta+it$ (including the conjugate zero $\beta-it$), and treat $\beta+it$ as simple.  This would increase the right-hand side \eqref{Eq:FirstLemma} and simplify the sum greatly, leaving only one term.  Thinking about this another way, the worst-case scenario is when $\beta+it$ is an isolated, simple zero.  This type of ``trivial'' estimation leads to decent zero-free regions.  However, we will continue to keep track of all the zeros, to allow for better optimizations later.

In any case,  when $\rho$ is distant from $z$, the terms in the sum in \eqref{Eq:FirstLemma} are quite small.  To see this, note that for real variables $x$ and $y$, we have the equality
\[
\Re \cot(x+iy) = \frac{\sin(2x)}{\cosh(2y)-\cos(2x)},
\]
which decreases exponentially as $y$ increases.  Thus, when a zero, $\rho$, is vertically distant from $z$, its contribution to the sum is exponentially negligible anyway.  This applies in particular to the conjugate zero $\beta-it$, since we know $t>100$.  (In fact, by \cite{PT} we know $t>3\cdot 10^{12}$.)  Moreover, this also tells us that the positive contribution $\Re(-h(1-z))$, arising from the pole at $s=1$, is negligible.  (Hereafter we will use big-O and little-o notations, which are taken with respect to $t\to \infty$.)  This term can easily be absorbed into other error terms.

It is worth pointing out that if we knew that there were indeed other zeros near $\beta+it$, this would improve the numerics in bounding $\beta$ away from $1$.  A similar idea is used to great effect in \cite{HB2}, in more general circumstances.

We next give an upper bound for the integral
\[
\frac{1}{4\eta}\int_{-\infty}^{\infty}\frac{\log|\zeta(z-\eta+\frac{2\eta iu}{\pi})|}{\cosh^2(u)}\, du.
\]
Recall that this integral is the contribution from the left side of the (now infinite) rectangle.  We assume $\eta>\Re(z-1)$, so we are penetrating the critical strip.  Bounds on $\zeta$ in this region arise from the methods of Korobov \cite{Korobov} and Vinogradov \cite{Vinogradov}.  These methods yield an upper bound of the form
\begin{equation}\label{Eq:KVBound}
\text{for every $T\geq 3$, }\ |\zeta(\sigma+iy)|\leq A T^{B(1-\sigma)^{3/2}}\log^{2/3}T\ \text{ if $1\leq |y|\leq T$ and $1/2\leq \sigma\leq 1$},
\end{equation}
where $A,B>0$ are some explicit constants.  Improving this bound is a fundamental problem, and represents the most difficult step in asymptotically improving the zero-free region for $\zeta$.  We will pass over these difficulties, and merely accept the bound.

\begin{remark}
We will leave the explicit constants $A$ and $B$ unspecified, for two reasons.  First, they continue to be improved.  Most recently, in \cite{Ford2} it was shown that the values of $A=76.2$ and $B=4.45$ will work, which was then improved to $A=62.6$ when $\sigma=1$ in \cite{Trudgian}.  Second, a small---but probably fixable---error in the literature has propagated through such numerics; this error was noticed by Kevin Ford and is explained in \cite{Patel}.
\end{remark}

In practice, $\eta$ will be extremely small, so assuming $\eta<\frac{1}{2}<\min(\frac{\pi}{4},\frac{1}{2}+\epsilon)$ costs us nothing.  For technical reasons, we assume $\eta\geq \epsilon+\frac{1}{t}$, so we may apply \cite[Lemma 3.4]{Ford} with $\sigma=1+\epsilon-\eta$ and $a=\frac{2\eta}{\pi}$ to obtain
\[
\frac{1}{4\eta}\int_{-\infty}^{\infty}\frac{\log|\zeta(1+\epsilon+it-\eta+\frac{2\eta iu}{\pi})|}{\cosh^2(u)}\, du<\frac{1}{2\eta}\left(\log A+B\eta^{3/2}\log t+\frac{2}{3}\log\log t\right).
\]
To optimize the right side, as a function of $t$, we will take
\begin{equation}\label{Eq:EtaAsymp}
\eta=C\left(\frac{\log\log t}{\log t}\right)^{2/3}
\end{equation}
for some positive constant $C$ that is to be chosen later.

Each of the previous simplifications appears in \cite{Ford}.  We end with one new observation; it concerns the integral along the right side of the rectangle.

\begin{lemma}\label{Lemma:NewkSums}
If $\eta>0$ and $\Re(z)>1$, then
\[
\frac{1}{4\eta}\int_{-\infty}^{\infty}\frac{\log|\zeta(z+\eta+\frac{2\eta i u}{\pi})|}{\cosh^2(u)}\, du = \sum_{k=1}^{\infty}-\Re\frac{\zeta'(z+2k\eta)}{\zeta(z+2k\eta)}.
\]
\end{lemma}
\begin{proof}
Apply \cite[Lemma 2.2]{Ford} to the function $f=\zeta$, but with the center point $z+2k\eta$ instead of $z$.  There are no zeros or poles of $\zeta(s)$ in (or near) the strip
\[
\Re(z)+(2k-1)\eta\leq \Re(s)\leq \Re(z)+(2k+1)\eta.
\]
We obtain
\[
-\Re\frac{\zeta'(z+2k\eta)}{\zeta(z+2k\eta)}= \frac{1}{4\eta}\int_{-\infty}^{\infty}\frac{\log|\zeta(z+(2k-1)\eta+\frac{2\eta iu}{\pi})|-\log|\zeta(z+(2k+1)\eta+\frac{2\eta iu}{\pi})|}{\cosh^2(u)}\, du.
\]
As $k$ ranges from $1$ to $m$, we see that the integral terms telescope, yielding
\[
\sum_{k=1}^{m}-\Re\frac{\zeta'(z+2k\eta)}{\zeta(z+2k\eta)} = \frac{1}{4\eta}\int_{-\infty}^{\infty}\frac{\log|\zeta(z+\eta+\frac{2\eta iu}{\pi})|-\log|\zeta(z+(2m+1)\eta+\frac{2\eta iu}{\pi})|}{\cosh^2(u)}\, du.
\]
As $\Re(s)\to \infty$, we have $\zeta(s)\to 1$ uniformly.  Thus, taking $m\to \infty$ gives the claimed equality.
\end{proof}

Putting this all together, we have:
\begin{prop}\label{Prop:Simplified}
Suppose that
\begin{itemize}
\item $\zeta$ satisfies \eqref{Eq:KVBound} for some positive constants $A$ and $B$,
\item $\eta$ satisfies \eqref{Eq:EtaAsymp} for some positive constant $C$,
\item $0<\epsilon<\eta-1/t$,
\item $z=1+\epsilon+it$, and
\item $t>100$.
\end{itemize}
Then we have
\begin{equation}
\begin{array}{rcl}
\displaystyle \sum_{k=0}^{\infty}-\Re\frac{\zeta'(z+2k\eta)}{\zeta(z+2k\eta)} & < & \displaystyle\frac{\pi}{2\eta}\sum_{\substack{\rho\, :\, \zeta(\rho)=0,\\ |\Re(z-\rho)|\leq \eta}} \Re \cot \left(\frac{\pi(\rho-z)}{2\eta}\right)\\[30pt]
& & \  \displaystyle +\left(\frac{BC^{3/2}+2/3}{2C}+o(1)\right)(\log t)^{2/3}(\log\log t)^{1/3}.
\end{array}
\end{equation}
\end{prop}

\section{Trigonometric inequalities}

A common component of some of the earliest proofs of zero-free regions for $\zeta$ in the critical strip involve the fact that
\[
3+4\cos(\vartheta)+\cos(2\vartheta)=2(1+\cos(\vartheta))^2\geq 0,
\]
which holds for all real values of $\vartheta$.  Other trigonometric inequalities also appear in the literature.  For example, in both \cite{HB} and \cite{Ford}, more complicated trigonometric inequalities are used the optimize computations.  To that end, let $b_0,b_1,\ldots, b_d$ be positive constants such that
\begin{equation}\label{Eq:TrigInequalityBs}
p(\vartheta)=\sum_{j=0}^{d} b_j \cos(j\vartheta)\geq 0
\end{equation}
for any real value of $\vartheta$.  We will attempt to optimize the constants $b_0,b_1,\ldots, b_d$ in a later section.

The connection between \eqref{Eq:TrigInequalityBs} and the problem at hand comes from the fact that
\[
-\frac{\zeta'(s)}{\zeta(s)}=\sum_{n=2}^{\infty}\frac{\Lambda(n)}{n^s}
\]
when $\Re(s)>1$, and so
\[
-\Re\frac{\zeta'(s)}{\zeta(s)} = \sum_{n=2}^{\infty}\frac{\Lambda(n)}{n^{\Re(s)}}\cos(\Im(s)\log n).
\]
Thus, for real parameters $x,y$, with $x>1$, we then have
\begin{equation}\label{Eq:AppliedTrigInequality}
\sum_{j=0}^{d}-b_j\Re\frac{\zeta'(x+ijy)}{\zeta(x+ijy)} =\sum_{n=2}^{\infty} \frac{\Lambda(n)}{n^{x}}p(y\log n)\geq 0.
\end{equation}

We use \eqref{Eq:AppliedTrigInequality} when $x=1+\epsilon+2k\eta$ (for each integer $k\geq 1$) and when $y=t$.  In that case, the $j=1$ summand can be upper bounded using Proposition \ref{Prop:Simplified}.

We want to similarly bound the terms when $j>2$.  By running through the argument of the previous section, with $1+\epsilon+it$ replaced by $1+\epsilon+ijt$ (with $j\geq 2$) we obtain essentially the same results.  Thus, we strengthen Proposition \ref{Prop:Simplified} to the following:

\begin{prop}\label{Prop:SimplifiedWithTrig}
Suppose that
\begin{itemize}
\item $\zeta$ satisfies \eqref{Eq:KVBound} for some positive constants $A$ and $B$,
\item $\eta$ satisfies \eqref{Eq:EtaAsymp} for some positive constant $C$,
\item $0<\epsilon<\eta-1/t$,
\item $p$ satisfies \eqref{Eq:TrigInequalityBs} for some positive constants $b_0,b_1,\ldots, b_d$,
\item $z_{j,k}=1+\epsilon+2k\eta+ijt$, and
\item $t>100$.
\end{itemize}
Then we have
\begin{equation}\label{Eq:NonMollified}
\begin{array}{rcl}
0 &\leq & \displaystyle \sum_{j=0}^{d}\sum_{k=0}^{\infty}-b_j\Re\frac{\zeta'(z_{j,k})}{\zeta(z_{j,k})}\\[20pt]
 & < & \displaystyle \sum_{k=0}^{\infty}-b_0\frac{\zeta'(z_{0,k})}{\zeta(z_{0,k})}+ \sum_{j=1}^{d}\frac{\pi b_j}{2\eta}\sum_{\substack{\rho\, :\, \zeta(\rho)=0,\\ \Re(z_{j,0}-\rho)\leq \eta}} \Re \cot \left(\frac{\pi(\rho-z_{j,0})}{2\eta}\right)\\[30pt]
& & \  \displaystyle +\left(\sum_{j=1}^{d}b_j\right)\left(\frac{BC^{3/2}+2/3}{2C}+o_d(1)\right)(\log t)^{2/3}(\log\log t)^{1/3}.
\end{array}
\end{equation}
\end{prop}

We finish this section by giving an additional bound on the sum $\sum_{k=1}^{\infty}-\frac{\zeta'(z_{0,k})}{\zeta(z_{0,k})}$.  This sum is the midpoint approximation of the integral
\[
\frac{1}{2\eta}\int_{1+\epsilon+\eta}^{\infty}-\frac{\zeta'(u)}{\zeta(u)}\, du = \frac{\log \zeta(1+\epsilon+\eta)}{2\eta} < \frac{\log \zeta(1+\eta)}{2\eta}.
\]
The midpoint approximation underestimates the integral, because the integrand is concave up.  Thus, by standard bounds on $\zeta$ along the real axis, and by \eqref{Eq:EtaAsymp}, we have
\[
-b_0\sum_{k=1}^{\infty}\frac{\zeta'(z_{0,k})}{\zeta(z_{0,k})} < b_0 \frac{\log \zeta(1+\eta)}{2\eta}=b_0 \frac{\log(1/\eta)+O(1)}{2\eta}<b_0\frac{\left(\frac{2}{3} + o(1)\right)(\log t)^{2/3}(\log\log t)^{1/3}}{2C}.
\]
This is precisely the bound achieved in \cite[Lemma 5.1]{Ford}, by different methods.  Now, combining this estimation with the last line of \eqref{Eq:NonMollified}, we see that to optimize the constant $C$ we need to minimize
\[
\left(\sum_{j=1}^{d}b_j\right)\frac{BC^{1/2}}{2} +\left(\sum_{j=0}^{d}b_j\right)\frac{1}{3C}.
\]
This is achieved by setting
\begin{equation}\label{Eq:Cdef}
C=\left(\frac{4\sum_{j=0}^{d}b_j}{3B\sum_{j=1}^{d}b_j}\right)^{2/3}.
\end{equation}
Taking $\epsilon\to 0$ everywhere in \eqref{Eq:NonMollified} except for the term $-b_0\frac{\zeta'(z_{0,0})}{\zeta(z_{0,0})}$, then under the same hypotheses as in Proposition \ref{Prop:SimplifiedWithTrig} we obtain
\begin{equation}\label{Eq:NonMollified2}
\begin{array}{rcl}
\displaystyle \sum_{j=0}^{d}\sum_{k=0}^{\infty}-b_j\Re\frac{\zeta'(z_{j,k})}{\zeta(z_{j,k})} & < & \displaystyle -b_0\frac{\zeta'(1+\epsilon)}{\zeta(1+\epsilon)} - \sum_{j=1}^{d}\frac{\pi b_j}{2\eta}\sum_{\substack{\rho\, :\, \zeta(\rho)=0,\\ 1-\Re\rho\leq \eta}} \Re \cot \left(\frac{\pi(1+ijt-\rho)}{2\eta}\right)\\
 & &\ \displaystyle +(1+o_d(1))\left(\frac{3}{4}B\left(\sum_{j=1}^{d}b_j\right)\left(\sum_{j=0}^{d}b_j\right)^{1/2}\log t\, (\log\log t)^{1/2} \right)^{2/3}.
\end{array}
\end{equation}
In the next section we will see how to take $\epsilon\to 0$ in the one remaining term where it appears.

\section{Mollified sums}

We are interested in estimating sums of the form
\[
\sum_{n=1}^{\infty} \frac{\Lambda(n)}{n^s}f(\log n),
\]
where $f$ is a nonnegative function.  The motivation for considering such mollified sums is well-explained in Sections 4--7 of \cite{HB2}, and the following is the optimal choice for $f$, as worked out in \cite[Section 6]{Ford}.

Let
\begin{equation}\label{Eq:LambdaDef}
\lambda=\frac{M}{(B\log t)^{2/3}(\log\log t)^{1/3}},
\end{equation}
where $M\geq 0$ will be determined later.  We will assume hereafter that $\zeta$ has no zero, $\rho$, in the region $1-\lambda\leq \Re \rho$ and $t-1\leq \Im\rho\leq dt+1$.  The results of other papers (or from the previous sections of this note) quickly shows that we can take $M>0$; this strict inequality will prove useful later.

Also, let $\theta$ be the unique solution of
\begin{equation}
\sin^2\theta=\frac{b_1}{b_0}(1-\theta\cot\theta),\ 0<\theta<\pi/2.
\end{equation}
Define the real function
\begin{equation}
g(u)=\begin{cases}
(\cos(u\tan \theta)-\cos\theta)\sec^2\theta & |u|<\theta/\tan\theta,\\
0 & \text{otherwise}.
\end{cases}
\end{equation}
Notice that $g$ is nonnegative and even, with finite support.  Put $w(u)=(g\ast g)(u)$, the convolution square of $g$, for $u\geq 0$.  Finally set
\begin{equation}\label{Eq:fdef}
f(u)=\lambda e^{\lambda u} w(\lambda u),\ \text{for }u\geq 0.
\end{equation}
Thus, $f$ is nonnegative, and has finite support $[0,2\theta/\tan\theta)$.  We have from \cite[Equation 6.6]{Ford},
\begin{equation}
f(0)=\lambda \sec^2\theta(\theta\tan \theta+3\theta\cot\theta-3)=O(\lambda).
\end{equation}

For notation ease, we will denote the Laplace transform of $f$ as
\begin{equation}\label{Eq:Fdef}
F(z)=\int_{0}^{\infty}f(y)e^{-zy}\, dy,
\end{equation}
and we will also make use of the auxiliary function
\begin{equation}\label{Eq:F0def}
F_0(z)=F(z)-\frac{f(0)}{z}.
\end{equation}
An important consequence of the choices above is that we have a bound
\begin{equation}\label{Eq:DBound}
|F_0(z)|\leq \frac{D}{|z|^2},
\end{equation}
where $D>0$ is a parameter that is independent of $z$ (but not $t$), and where $\Re z\geq 0$ and $|z|\geq \eta$.  From \cite[Equation 7.6]{Ford} we get that \eqref{Eq:DBound} holds with
\[
D=O(\lambda^2)=O\left(\frac{1}{(\log t)^{4/3}(\log\log t)^{2/3}}\right).
\]

This ultimately leads us to the following:

\begin{thm}\label{Thm:AlmostDone}
Suppose that
\begin{itemize}
\item $t>10000$,
\item $\eta$ satisfies \eqref{Eq:EtaAsymp}, where $C$ satisfies \eqref{Eq:Cdef},
\item $\lambda$ satisfies \eqref{Eq:LambdaDef}, for some constant $M>0$, and satisfies $\lambda<\frac{\eta}{250}$,
\item $p$ satisfies \eqref{Eq:TrigInequalityBs} for some positive constants $b_0,b_1,\ldots, b_d$,
\item $\theta,g,w,f,F,$ and $F_0$ are defined as at the beginning of this section,
\item $\zeta$ satisfies \eqref{Eq:KVBound} for some $A>6.5$ and some $B>0$,
\item $\zeta$ has a root at $\beta+it$, with $1-\beta<\eta$, and
\item $\zeta$ has no root $\rho$ in the region where $1-\lambda\leq \Re\rho$ and $t-1\leq \Im \rho\leq dt+1$.
\end{itemize}
Then
\begin{equation}\label{Eq:BoundOnVc}
\begin{array}{rcl}
0 & \leq & \displaystyle b_0F(0)-b_1\left(F_0(1-\beta)+f(0)\frac{\pi}{2\eta}\cot\left(\frac{\pi(1-\beta)}{2\eta}\right)\right)\\
& & \displaystyle + Mw(0)\left(\frac{3}{4}\left(\sum_{j=1}^{d}b_j\right)\left(\sum_{j=0}^{d}b_j\right)^{1/2}\right)^{2/3} + o_p(1).
\end{array}
\end{equation}
\end{thm}
\begin{proof}
If $\Re s>1$ and $0\leq \Im s=O(t)$, then by \cite[Lemma 4.5]{Ford} we obtain the equality
\begin{equation}\label{Eq:FirstMollifiedSum}
K(s):=\sum_{n=2}^{\infty}\frac{\Lambda(n)}{n^s}f(\log n) = -f(0)\frac{\zeta'(s)}{\zeta(s)}-\sum_{\substack{\rho\, :\, \zeta(\rho)=0,\\ \Re(\rho)>0}}F_0(s-\rho)+F_0(s-1)+o(1).
\end{equation}
(The hypothesis that $f$ is continuous from the right at $0$ was inadvertently dropped as a hypothesis in \cite[Lemma 4.5]{Ford}.  Also, the constant 10.8 in the proof should be replaced by $11.13$, which also affects later constants.)  Thus, again taking $z_{j,k}=1+\epsilon+2k\eta+ijt$, we have
\begin{equation*}
\begin{array}{rcl}
0 & \leq & \displaystyle \sum_{j=0}^{d}b_j \Re \left(K(z_{j,0})+f(0)\sum_{k=1}^{\infty}-\frac{\zeta'(z_{j,k})}{\zeta(z_{j,k})} \right) \\
 & = & \displaystyle \sum_{j=0}^{d}b_j \Re\left(
 f(0)\left(\sum_{k=0}^{\infty}-\frac{\zeta'(z_{j,k})}{\zeta(z_{j,k})}\right) -\sum_{\substack{\rho\, :\, \zeta(\rho)=0,\\ \Re(\rho)>0}}F_0(z_{j,0}-\rho)+F_0(z_{j,0}-1)\right)+o_p(1).
\end{array}
\end{equation*}
When $j\geq 1$, the terms $F_{0}(z_{j,0}-1)$ are $o(1)$ by \eqref{Eq:DBound}, but when $j=0$ the corresponding term is $F_0(\epsilon)$, which is nontrivial.  By \eqref{Eq:NonMollified2} we obtain
\begin{equation*}
\begin{array}{rcl}
0 & \leq  & \displaystyle b_0\left[ -f(0)\frac{\zeta'(1+\epsilon)}{\zeta(1+\epsilon)}+F_0(\epsilon) -\sum_{\substack{\rho\, :\, \zeta(\rho)=0,\\ \Re(\rho)>0}}F_0(1+\epsilon-\rho)\right]\\
& & \displaystyle + \sum_{j=1}^{d}b_j \Re\left(-\frac{\pi f(0)}{2\eta}\sum_{\substack{\rho\, :\, \zeta(\rho)=0,\\ 1-\Re\rho\leq \eta}} \cot \left(\frac{\pi(1+ijt-\rho)}{2\eta}\right) -\sum_{\substack{\rho\, :\, \zeta(\rho)=0,\\ \Re(\rho)>0}}F_0(1+ijt-\rho)\right)\\
 & & \displaystyle +f(0)\left(\frac{3}{4}B\left(\sum_{j=1}^{d}b_j\right)\left(\sum_{j=0}^{d}b_j\right)^{1/2}\log t\, (\log\log t)^{1/2} \right)^{2/3}+o_p(1).
 \end{array}
\end{equation*}
The quantity within the square braces is $K(1+\epsilon)+o(1)$.  It is finally possible to take $\epsilon\to 0$; by \cite[Lemma 4.6]{Ford} we have $K(1)\leq F(0)+o(1)$.  Moreover, by \cite[Equations (6.6) and (6.8)]{Ford} we have
\begin{equation}
F(0)=2\tan^2\theta+3-3\theta(\tan \theta+\cot\theta)
\end{equation}
as well as
\begin{equation}
\frac{f(0)}{\lambda} = w(0)=\sec^2\theta(\theta\tan\theta + 3\theta\cot\theta-3).
\end{equation}
So, using \eqref{Eq:LambdaDef}, our inequality simplifies to
\begin{equation}
\begin{array}{rcl}
0 & \leq & \displaystyle \sum_{j=1}^{d}b_j\Re\left(-\frac{\pi f(0)}{2\eta}\sum_{\substack{\rho\, :\, \zeta(\rho)=0,\\ 1-\Re\rho\leq \eta}}  \cot \left(\frac{\pi(1+ijt-\rho)}{2\eta}\right) -\sum_{\substack{\rho\, :\, \zeta(\rho)=0,\\ \Re(\rho)>0}}F_0(1+ijt-\rho)\right)\\
& & \displaystyle + b_0 F(0) + Mw(0)\left(\frac{3}{4}\left(\sum_{j=1}^{d}b_j\right)\left(\sum_{j=0}^{d}b_j\right)^{1/2}\right)^{2/3}+o_p(1).
\end{array}
\end{equation}
The terms $F_0(1+ijt-\rho)$, where $|1+ijt-\rho| \geq \eta$, can be bound by \cite[Lemma 4.3]{Ford} in conjunction with \eqref{Eq:DBound}.  Their total contribution is $o_p(1)$.  Also, we can drop the terms from the summation involving cotangents where $|1+ijt-\rho|\geq \eta$, since (as we mentioned previously) such terms make a negative contribution.

Set $c=\frac{\pi\lambda}{2\eta}$ and let $z=\frac{\pi}{2\eta}(1+ijt-\rho)$, where $\rho$ is any of the remaining zeros where $|1+ijt-\rho|\leq \eta$.  In particular, we have $|z|\leq \frac{\pi}{2}$.  Set
\[
V_c(z) = f(0)\frac{\pi}{2\eta}\cot\left(\frac{\pi(1+ijt-\rho)}{2\eta}\right)+F_0(1+ijt-\rho).
\]
In \cite{Ford}, it is established that
\[
-\Re V_c(z)\leq O(c^2)=O\left(\frac{1}{(\log\log t)^2}\right).
\]
(See equation (7.7) in that paper.  Note that a lot more works goes into the actual bound, so as to give an explicit bound.)  By \cite[Lemma 4.2]{Ford} the number of such zeros is $O_p(\log\log t)$,so we may ignore each of the remaining terms.  However, we will not drop the single term when $\rho=\beta+it$.  The resulting inequality is exactly \eqref{Eq:BoundOnVc}.
\end{proof}

Now, our immediate goal is to estimate
\begin{multline*}
-b_1\left(F_0(1-\beta)+f(0)\frac{\pi}{2\eta}\cot\left(\frac{\pi(1-\beta)}{2\eta}\right)\right)\\ = b_1\left(-F(1-\beta) + f(0)\left(\frac{1}{1-\beta}-\frac{\pi}{2\eta}\cot\left(\frac{\pi(1-\beta)}{2\eta}\right)\right)\right).
\end{multline*}
By \cite[Lemma 4.4]{Ford}, this quantity is
\[
\leq -b_1 F(1-\beta) + b_1\frac{f(0)}{\eta^2}(1-\beta)=-b_1 F(1-\beta)+o_p(1),
\]
using the assumption that $1-\beta\leq \eta$.  (In the contrary case, we have a bound better than the one we were hoping for.)

So, our immediate goal is now to estimate
\[
b_0F(0)-b_1F(1-\beta).
\]
This quantity has been optimized (subject to all the other constraints) via the choice of $\theta$, and the definitions of $f$ and $F$.  Letting $W$ be the Laplace transform of $w$, we have by equation (7.13) in \cite{Ford}, and the preceding discussion, that
\[
bF(0)-b_1F(1-\beta) =-b_0w(0)\cos^2\theta - b_1W'(0)\left(\frac{1-\beta}{\lambda}-1\right).
\]
(Right before that equation it is claimed that ``$W(x)$ and $W'(x)$ are both decreasing''.  What was meant is that $W(x)$ and $-W'(x)$ are decreasing.)  With the help of Mathematica, we compute that
\begin{equation}
-W'(0)=\frac{1}{3} \csc \theta ((15 - 12 \theta^2 + \theta (-15 + 4 \theta^2) \cot\theta) \csc\theta + 3 \theta \sec\theta),
\end{equation}
although we will see shortly that this value is irrelevant.  Together with Theorem \ref{Thm:AlmostDone}, the previous works leads us to want to optimize the bound
\begin{equation}\label{Eq:MinimizeM}
\frac{1-\beta}{\lambda}-1
\geq \frac{1}{-b_1W'(0)}\left(b_0w(0)\cos^2\theta-Mw(0)\left(\frac{3}{4}\left(\sum_{j=1}^{d}b_j\right)\left(\sum_{j=0}^{d}b_j\right)^{1/2}\right)^{2/3}\right).
\end{equation}
The last bulleted condition in Theorem \ref{Thm:AlmostDone} tells us that $1-\beta\geq \lambda$, and so the smallest that the right hand side of \eqref{Eq:MinimizeM} can be is $0$.  Thus, the value of $M$ that optimizes the inequality is
\begin{equation}\label{Eq:ValueOfM}
M=\frac{b_0\cos^2\theta}{\left(\frac{3}{4}\left(\sum_{j=1}^{d}b_j\right)\left(\sum_{j=0}^{d}b_j\right)^{1/2}\right)^{2/3}}.
\end{equation}

It is important, here, to point out that this inequality is self-improving if $M$ is chosen smaller than this value.  Recall that we already chose $M>0$ such that there is no zero, $\rho$, with $1-\lambda\leq \Re\rho$, and with $\Im \rho = O_{p}(t)$.  If $M$ is any smaller than the value given in \eqref{Eq:ValueOfM}, then the inequality \eqref{Eq:MinimizeM} immediately shows us that we could have increased $M$ anyway.

The quantity on the right side of \eqref{Eq:ValueOfM} is defined solely in terms of $b_0,b_1,\ldots, b_d$.  (It does involve $\theta$, which is also defined solely in terms of $b_0$ and $b_1$.)  We can numerically search for the optimal solution.  When $d=4$, the best value for $M$ is approximately $0.05507$, as in \cite{Ford}.  When $d=5$, there is a slightly better value of
\[
M=0.055127\ldots
\]
that is obtained using the trigonometric polynomial
\[
p(\vartheta)=(1+\cos\vartheta)(0.8652559\ldots + \cos\vartheta)^2(0.1974476\ldots+\cos\vartheta)^2.
\]
The numerics are apparently not improved by using even higher degree polynomials.

\section*{Acknowledgements}

This work was partially supported by a grant from the Simons Foundation (\#963435 to Pace P.\ Nielsen).

\providecommand{\MR}{\relax\ifhmode\unskip\space\fi MR }
\providecommand{\MRhref}[2]{%
  \href{http://www.ams.org/mathscinet-getitem?mr=#1}{#2}
}
\providecommand{\href}[2]{#2}

\end{document}